\documentclass[12pt]{article}

\usepackage[utf8]{inputenc}
\usepackage{amsmath,amsfonts,amsthm,amssymb}
\usepackage{fullpage}

\newtheorem{theorem}{Theorem}[section]
\newtheorem{lemma}[theorem]{Lemma}
\newtheorem{proposition}[theorem]{Proposition}
\newtheorem{corollary}[theorem]{Corollary}

\title{Imaginary quadratic fields with 2-class group of type $(2,2^\ell)$}
\date{}
\author{Adele Lopez}

\begin{document}
 \maketitle
\abstract{We prove that for any given positive integer $\ell$ there are infinitely many imaginary quadratic fields with 2-class group of type $(2,2^\ell)$, and provide a lower bound for the number of such groups with bounded discriminant for $\ell\ge2$. This work is based on a related result for cyclic 2-class groups by Dominguez, Miller and Wong, and our proof proceeds similarly. Our proof requires introducing congruence conditions into Perelli's result on Goldbach numbers represented by polynomials, which we establitish in some generality. 
  }

\section{Introduction}

Since the time of Gauss, mathematicians have been interested in imaginary quadratic fields and their ideal class groups. Gauss himself provided much of the framework for such studies with the development of his genus theory for binary quadratic forms. Later developments by R\`edei~\cite{R} and others such as Hasse~\cite{H} have given algorithms which compute the 2-class group from the discriminant of the imaginary quadratic field, which reveal much underlying structure.

However, not much work has been done in the converse direction of computing imaginary quadratic fields with a given 2-class group. Recently, Dominguez, Miller and Wong~\cite{DMW} proved that there are infinitely many imaginary quadratic fields with any given cyclic 2-class group. They determined a set of criteria that the discriminant of such a field would have to satisfy, and then used the circle method to show that there are infinitely many integers satisfying those criteria.

In their paper, Dominguez, Miller and Wong asked whether similar results could be found for other types of groups. We use the same technique to prove that for any given positive integer $\ell$,  there are infinitely many imaginary quadratic fields with a 2-class group with type $(2,2^\ell)$.

There has also been work on finding lower bounds for the number class groups of imaginary quadratic fields with elements of a given order, for example, Murty~\cite{Murty} found that for $g\ge2$: $$\left|\left\{ d\le X: \text{Cl}(-d)\text{ contains an element of order } g\right\}\right|\gg \frac{X^{\frac{1}{2}+\frac{1}{g}}}{\log^2 X}.$$ We have been able to achieve a similar lower bound for the class groups under consideration.

Also, Balog and Ono~\cite{BalogOno} have proven a similar theorem that gives certain conditions for how often the $\ell$-torsion of the ideal class group of an imaginary quadratic field is non-trivial. Their technique is similar to ours, relying on the circle method. In our case, however, we are interested in specific subgroups of the class group. 

Some results for imaginary quadratic fields with 2-class group of this type have been studied. For example, Benjamin, Lemmermyer and Snyder~\cite{BLS} proved the following. For an imaginary quadratic number field $K$, let $K^1$ be the Hilbert 2-class field of $K$. Then if the 2-class group of $K^1$ is cyclic, the 2-class group of $K$ has type $(2,2^\ell)$.

Using genus theory, and other algebraic considerations, we establish a sufficient set of criteria for an imaginary quadratic field to have 2-class group of type $(2,2^\ell)$. 
 \begin{proposition}
 If $w=3m^2$ where $m$ is an odd integer, $p_1\equiv 11 \bmod 24$, $p_2\equiv 7 \bmod 24$, and $p_1+p_2=2w^{2^{\ell-1}}$ with $\ell\ge 2$, then the 2-class group of $\mathbb Q(\sqrt{-p_1p_2})$ has type $(2,2^\ell)$.
 \end{proposition}

We now need to prove that there are infinitely many primes satisfying this criteria. Steven J. Miller asked the author if there was a more general way to introduce congruence conditions into Perelli's result on Goldbach numbers represented by polynomials~\cite{P}, which would imply that there are infinitely many such primes. We have found such a general theorem, which we prove using the circle method.
\begin{theorem}
Let $m$ be an even positive integer, and let $s_1$ and $s_2$ be relatively prime to $m$. Let $F\in \mathbb Z[x]$ be a polynomial with degree $k>0$ and positive leading coefficient that takes on an even value congruent to $s_1+s_2$ modulo $m$. Then there are infinitely many pairs of primes congruent to $s_1$ and $s_2$ modulo $m$, which sum to $F(n)$ for some $n$.
\end{theorem}

Putting these together, we prove our main theorem.

\begin{theorem}
There are infinitely many imaginary quadratic fields with 2-class group of type $(2,2^\ell)$, for any positive integer $\ell$. 

In particular, if $\text{Cl}_2(-d)$ is the 2-class group of $\mathbb Q(\sqrt{-d})$ and $\ell>1$, then
$$\left|\left\{ d\le X, \text{Cl}_2(-d)\cong(2,2^\ell)\right\}\right|\gg \frac{X^{\frac{1}{2}+\frac{1}{2\cdot2^\ell}}}{\log^2 X}.$$

\end{theorem}

The author thanks the Mt.~Holyoke REU (NSF grant DNS-0849637), the BYU Mathematics department, the BYU Office of Research and Creative Activities, and the National Science Foundation (NSF grant 1000132486) for their generous support. She also thanks Siman Wong, Steven J. Miller, Giuliana Davidoff, and Paul Jenkins, for their guidance with this project, and she thanks Ken Ono, Robert Lemke-Oliver, and Darrin Doud for their helpful comments and suggestions.

 \section{Prescribing the 2-class group}

We begin by using some algebraic number theory and Gauss' genus theory to find sufficient conditions for the discriminant of an imaginary quadratic field $\mathbb Q(\sqrt{-d})$ to have a 2-class group of the desired form. 

The following lemma is originally due to Ankeny and Chowla~\cite[Theorem 1]{AC}, modified slightly. It should be noted that Soundarajan~\cite{Sound} has significantly improved the lower bound on the number of such fields.  It remains to be seen whether one can use his improved result to improve a result on the type of question considered in this paper. 

 \begin{lemma}\label{AC}
 Fix $m>1$ and let $d=w^{2m}-x^2 >0$ with $w,x\in\mathbb Z$, $x$ even, $(x,w)=1$ and $0<x\le w^m-4$. Then the class group of $\mathbb Q(\sqrt{-d})$ has an element of order $2m$.  
 \end{lemma}
 \begin{proof}

Consider the ideals $(x+\sqrt{-d})$ and $(x-\sqrt{-d})$. We claim these are coprime. If not, there is some ideal $\mathfrak p$ of $\mathcal O_{-d}$ which divides both of them, and hence divides both their product, $(x^2+d)=(w^{2m})$ and  the ideal $(2x)$, as $x+\sqrt{-d}+x-\sqrt{-d}=2x$. However, $w$ and $2x$ are relatively prime, so this is impossible. Since $(x+\sqrt{-d})(x-\sqrt{-d})=(x^2+d)=(w^{2m})$, the factor $(x+\sqrt{-d})$ is equal to $J^{2m}$ for some ideal $J$ of $\mathcal O_{-d}$.
 
 Now suppose that $J$ has order less than $2m$, so that for some $0<n\le m$, the ideal $J^n$ is principal and thus $J^n=(u+v\sqrt{-d})$  for some $u,v\in \frac{1}{2}\mathbb Z$. We note that $v$ cannot be 0, since $J^{2m}$ contains non-real elements. Since $v\ne 0$, then $\frac{d}{4}\le u^2+v^2d= \text{Norm}(J^n)$. Since $n\le m$, and  $\text{Norm}(J^{2m})=w^{2m}$, we have that $d\le 4w^m$. But $0<d=w^{2m}-x^2$, and so $w^{2m}-4w^m\le x^2$, thus $(w^m-2)^2\le x^2+4< (x+2)^2$, which contradicts our condition on $x$. Therefore, the ideal class of $J$ has order exactly $2m$.

 \end{proof}

Applying this with genus theory, we get the following corollary.

 \begin{corollary}
 
 Fix $\ell\ge 1$, and let $w$ be an odd integer such that $2w^{2^{\ell-1}}$ is the sum of two distinct primes $p_1,p_2\ge 5$. Then the 2-Sylow class group of $\mathbb Q(\sqrt{-p_1p_2})$ has type $(2^\upsilon,2^{\ell'})$, where $\ell'\ge \ell$.
 
 \end{corollary}
 \begin{proof}
 Since $w$ is odd, $p_1p_2\equiv 1\bmod 4$. Thus by genus theory, we have a 2-class group of type $(2^\upsilon,2^{\ell'})$. To show that $\ell'\ge \ell$, we apply the above lemma to our primes, which we write as $w^{2^{\ell-1}}\pm x$, so that $p_1p_2=w^{2^\ell}-x^2$. Since $p_1,p_2\ge 5$, the condition on $x$ is satisfied, so $\ell'\ge \ell$.
 
 \end{proof}
 
This allows us to describe a condition for the discriminant of an imaginary quadratic field so that it has a 2-class group of the desired type, which is our version of ~\cite[Lemma 2.3]{DMW}

 \begin{proposition}\label{alg}
 If $w=3m^2$ where $m$ is an odd integer, $p_1\equiv 11 \bmod 24$, $p_2\equiv 7 \bmod 24$, and $p_1+p_2=2w^{2^{\ell-1}}$ with $\ell\ge 2$, then the 2-class group of $\mathbb Q(\sqrt{-p_1p_2})$ has type $(2,2^\ell)$.
 \end{proposition}
 \begin{proof}
 
 By the corollary, we know that the 2-class group is of the form $(2^\upsilon,2^{\ell'})$, where $\ell'\ge \ell$. 

We now wish to show that the group has the desired form, i.e.  $\upsilon=1$ and $\ell'=\ell$. First we will consider the group abstractly, and consider what properties the 3 elements of order 2 must have for $\upsilon=1$ and $\ell'=\ell$. Then we will use our specifications to $w$, $p_1$, and $p_2$ given to us along with Hasse's fundamental criterion to show that our ideal classes with order 2 indeed have the desired properties.

We let $J$ be the ideal from Lemma~\ref{AC} with norm $w$ and in an ideal class with order exactly $2^\ell$. We let $A$, $B$, and $C$ be representatives of each of the ideal classes with order 2. Considering the ideal class group $(2^\upsilon,2^{\ell'})$, we note that it can be generated by two elements $a$ and $b$ such that $a^{2^{\ell'}}=b^\upsilon=1$. The three elements of order 2 are then $a^{2^{\ell'-1}}$, $b^{2^{\upsilon-1}}$, and $a^{2^{\ell'-1}}b^{2^{\upsilon-1}}$ respectively. 

Since $\ell\ge 2$, if $\upsilon\ge 2$ as well, that implies that all the ideal classes of $A$, $B$ and $C$ are square. So we must show that one of ideal classes $A$, $B$ or $C$ is non-square to show that $\upsilon=1$. 

Assuming $\upsilon=1$, we now consider the ideal class of $J$. We know that $J=a^rb^s$ for some $r,s\in\mathbb Z$. We will also consider the ideal classes of $JA$, $JB$ and $JC$, which are $a^{2^{\ell'-1}+r}b^s$, $a^{r}b^{s+1}$, and $a^{2^{\ell'-1}+r}b^{s+1}$ respectively. 

Suppose three of the ideal classes of $J$, $JA$, $JB$, and $JC$ are non-square. So one of $J$ and $JA$ is a non-square, thus if $s$ is even, $r$ must be odd. Similarly, one of $JB$ and $JC$ is non-square, so if $s$ is odd, $r$ must be odd. Hence $r$ is odd. Now $J$ has order $2^\ell$, so since $\ell\ge 2$, $(a^rb^s)^{2^{\ell}}=a^{r2^{\ell}}b^{s2^{\ell}}=a^{r2^{\ell}}=1$ and thus $a$ has order $r2^\ell$.  But this is a contradiction if $\ell'> \ell$ since $r$ is odd and $a$ has order $2^{\ell'}$. Therefore $\ell=\ell'$ if three of the ideal classes of $J$, $JA$, $JB$ and $JC$ are non-square.

We now will show that the ideal classes under consideration are indeed not square. To do this, we use 
 Hasse's fundamental criterion~\cite{H} which says that for an ideal $\mathfrak a \subset \mathcal O_{-D}$, the ideal class of $\mathfrak a$ is a square iff
 $$\left(\frac{\text{Norm}(\mathfrak a),-D}{p}\right)=1 \text{ for every prime } p|D.$$ In our case, $-p_1p_2\equiv -1 \bmod 4$, so $D=4p_1p_2$.

Let $p$ be the smaller of $p_1$ and $p_2$. By genus theory, $(2,1+\sqrt{-D}), (p,\sqrt{-D}), (2,1+\sqrt{-D})(p,\sqrt{-D})$ are in the three distinct ideal classes with order 2 (see, for example~\cite[Prop.~3.3 + Theorem~7.7]{CX}). Respectively, these ideals have norms $2, p$, and $2p$.

   We note that $\left(\frac{w}{p_1}\right)=\left(\frac{3}{p_1}\right)=1$ since $p_1\equiv 11 \bmod 12$, and that $p_2 \equiv 2w^{2^{\ell-1}}\bmod p_1$, so $\left(\frac{p_2}{p_1}\right)=\left(\frac{2}{p_1}\right)=-1$. Then by quadratic reciprocity $\left(\frac{p_1}{p_2}\right)=1$, since $p_1\equiv p_2\equiv 3 \bmod 4$.
 
 We now calculate the appropriate Hilbert symbols:
 
 $$\left(\frac{2,-D}{2}\right) = (-1)^{\omega(-p_1p_2)}=-1 \text{ since } p_1p_2\equiv 5\bmod 8,$$
 
 $$\left(\frac{w,-D}{2}\right) = (-1)^{1\cdot1}=-1 \text{ since } w\equiv-p_1p_2\equiv 3\bmod 4,$$

Since we're not sure which prime $p$ is, we'll establish the required calculations for both cases. 

 $$\left(\frac{p_1w,-D}{p_1}\right) = (-1)^{1}\left(\frac{w}{p_1}\right)\left(\frac{-4p_2}{p_1}\right)=\left(\frac{-1}{p_1}\right)\left(\frac{p_2}{p_1}\right)=-1 \text{ since } p_1\equiv 3\bmod 4,$$
 
 $$\left(\frac{p_2w,-D}{p_1}\right) = (-1)^{0}\left(\frac{p_2w}{p_1}\right)=\left(\frac{p_2}{p_1}\right)\left(\frac{w}{p_1}\right)=-1.$$
 
In either case, we have

$$\left(\frac{2pw,-D}{2}\right) = (-1)^{1\cdot1}=-1 \text{ since } -p_1p_2\equiv 3\bmod 8,$$

 Thus, none of the aforementioned ideal classes are square, so the 2-class group has type $(2,2^\ell)$.

\end{proof}

The above lemma requires that $\ell\ge2$. For the $\ell=1$ case, the full theorem is much simpler to prove.

 \begin{proposition}\label{k1}
There are infinitely many imaginary quadratic fields with 2-class group with type $(2,2)$.
 \end{proposition}
 \begin{proof}

Let $p_1\equiv 3 \bmod 8$ where $p_1$ is a prime, and let $p_2\equiv 7 \bmod 8p_1$. By Dirichlet's theorem, there are clearly an infinite number of such pairs of primes. Consider the imaginary quadratic field $\mathbb Q(\sqrt{-D})$, where $D=4p_1p_2$.

Since $p_1<p_2$, genus theory gives us that $(2,1+\sqrt{-D})$, $(p_1,\sqrt{-D})$, $(2,1+\sqrt{-D})(p_1,\sqrt{-D})$ are in the three distinct ideal classes with order 2. Respectively, these ideals have norms $2, p_1$, and $2p_1$. 

To prove that $\mathbb Q(\sqrt{-D})$ has a 2-class group of the desired type, we only need show that the above three ideal classes are non-square, which we can do by using Hasse's criterion. 

Calculating the appropriate Hilbert symbols, 

 $$\left(\frac{2,-D}{2}\right) = (-1)^{\omega(-p_1p_2)}=-1 \text{ since } -p_1p_2\equiv 3\bmod 8,$$

 $$\left(\frac{p_1,-D}{p_2}\right) = \left(\frac{p_1}{p_2}\right)=-1 \text{ by quadratic reciprocity,}$$

 $$\left(\frac{2p_1,-D}{p_2}\right) = \left(\frac{2p_1}{p_2}\right)=-1 \text{ since } p_2\equiv 7\bmod 8.$$

Thus, the field has a 2-class group of the required type.

\end{proof}

 \section{Circle method}

\subsection{Overview}
We now use the circle method to show that there are infinitely many pairs of primes satisfying the conditions in Proposition~\ref{alg}. We will do this by modifying Perelli's proof on Goldbach numbers represented by polynomials~\cite{P} to handle arbitrary congruence conditions. 

We will consider congruences modulo a positive integer $m$, with our primes $p_1,p_2$ equivalent to $s_1$ and $s_2$ modulo $m$, respectively. We want to find infinitely many pairs of primes whose sum is represented by a given polynomial $F\in\mathbb Z[x]$. Since there aren't very many primes dividing $m$, we restrict our consideration to  $s_1$ and $s_2$ be relatively prime to $m$. We will also require that the polynomial represents at least one value that is congruent to $s_1+s_2$ modulo $m$, and that it has degree $k\ge1$. 

We let $N$ be a sufficiently large positive integer, and we define $L=\log N$ and $P=L^B$, where $B$ is a positive constant. We'll take $n$ satisfying $N^{1/k}\le n \le N^{1/k}+H$ for some $H\le N^{1/k}$. So if $F(x)=a_kx^k+\cdots + a_0$, then $F(n)$ will be on the order of $c_0N$, where $c_0$ is a non-zero constant. If we restrict to primes smaller than $N$, then we will not have $F(n)=p_1+p_2$. So we will take primes up to $N$ times a non-zero constant $c_1$ to ensure we have enough room for solutions.

To apply the circle method to our problem, we will use the function $$f_s(\alpha)=\sum_{\substack{p\le c_1N \\ p\equiv s \bmod m}}(\log p)e(\alpha p),$$ where $e(x)=e^{2\pi i x}$. This will hold the desired information about the prime numbers equivalent to an arbitrary $s$ modulo $m$ that we wish to consider. We will consider also consider the related function $f_S(\alpha)=f_{s_1}(\alpha)+f_{s_2}(\alpha)$. By integrating it in the following manner, we are able to perform a weighted count of the number of such primes which sum to a given number $n$:

$$R_S(n)=\int_{[0,1]}f_S(\alpha)^2e(-\alpha n)d\alpha=\sum_{\substack{p_1,p_2\le c_1N\\ p_1+p_2 =n \\ p_1,p_2\equiv s_1,s_2\bmod m}}\log p_1 \log p_2.$$

This also counts pairs of primes both congruent to $s_1$ or $s_2$, but we can avoid counting these pairs with some basic congruence arguments. If we find a positive lower bound for this integral, there must be at least one such representation. Our focus will thus be on approximating and bounding this integral sufficiently well to prove that we have infinitely many such pairs of primes. Perelli's argument for Goldbach numbers represented by polynomials~\cite{P} does the bulk of this work, so we will follow his argument closely. Our modifications will involve restricting the congruence classes of the primes, to ensure that they have the desired properties. 
 
It is easier to split our integral into major arcs near rational points, and minor arcs everywhere else. Since $e(\alpha)$ has period 1, it does not matter which unit interval we integrate over, so we will choose the interval $(PN^{-1},1+PN^{-1})$ for convenience. For $1\le a \le q \le P$ with $(a,q)=1$, define
 
 $$\mathfrak M'(q,a)=\left\{\alpha:\left|\alpha-\frac{a}{q}\right|\le PN^{-1}\right\}$$ as the major arc centered at $\frac{a}{q}$. $\mathfrak M$ will denote the union of all the major arcs. Since $N$ is large, the major arcs are disjoint, and lie in $(PN^{-1},1+PN^{-1}].$ We define the minor arcs $\mathfrak m =(PN^{-1},1+PN^{-1}]\backslash \mathfrak M$. 

We will demonstrate that a certain sum $\mathfrak S_S(F(n))$ converges, and we'll use it to approximate the contribution from the major arcs.  Then we'll bound the minor arcs, and combine these to prove the following theorem. 

\begin{theorem}~\label{estimate}
Let $m$ be a positive integer and let $s_1$ and $s_2$ be relatively prime to $m$. 

Let $F\in \mathbb Z[x]$ be a polynomial with degree $k>0$ and positive leading coefficient, and let $L= \log N$, $A,\epsilon\ge 0$, and $H$ such that $N^{1/(3k)+\epsilon}\le H \le N^{1/k-\epsilon}$.

Then $$\sum_{N^{1/k}\le n \le N^{1/k}+H}|R_{S}(F(n))-F(n)\mathfrak S_{S} (F(n))|^2\ll HN^2L^{-A}.$$

\end{theorem}
\begin{proof}

We may assume following Perelli~\cite[\S 2]{P} that $H=N^{1/(3k)+\epsilon}$, that $\epsilon>0$ is sufficiently small, that $A>0$ is sufficiently large, and that $N\ge N_0(A,\epsilon)$ is a large constant. 
Now,
\begin{eqnarray*}&&\sum_{N^{1/k}\le n \le N^{1/k}+H}|R_{S}(F(n))-F(n)\mathfrak S_{S} (F(n))|^2\\
&=& \sum_{N^{1/k}\le n \le N^{1/k}+H}\left| \int_{\mathfrak M}f_S(\alpha)^2e(-F(n)\alpha)d\alpha+\int_{\mathfrak m}f_S(\alpha)^2e(-F(n)\alpha)d\alpha -F(n)\mathfrak S_S (F(n))\right|^2\\
&\le&\sum_{N^{1/k}\le n \le N^{1/k}+H}\left| \int_{\mathfrak M}f_S(\alpha)^2e(-F(n)\alpha)d\alpha -F(n)\mathfrak S_S(F(n))\right|^2\\
&+&\sum_{N^{1/k}\le n \le N^{1/k}+H}\left|\int_{\mathfrak m}f_S(\alpha)^2e(-F(n)\alpha)d\alpha\right|^2 =\sum_{\mathfrak M} + \sum_{\mathfrak m}.
\end{eqnarray*}

From Theorem~\ref{major} given in the following section, we have $$\sum_{\mathfrak M}\ll HN^2L^{-2B+c}+HN^2L^{-A},$$ where $c>0$ is a suitable constant depending on $m$, $F$ and $N$. From Theorem~\ref{minor}, we have that $$\sum_{\mathfrak m}\ll  HN^2L^{-B}$$ for sufficiently large $B$ depending on $k$.

Hence, by choosing $B$ to be sufficiently large in terms of $A$ and $k$, we can absorb everything into the $HN^2L^{-A}$ term, proving the theorem. 
\end{proof}

From this we can prove the following corollary.

\begin{corollary}~\label{cor2}

Let $A,\epsilon>0$, and let $H$ such that $N^{1/(3k)+\epsilon}\le H \le N^{1/k-\epsilon}$.   Then for almost all $n\in[N^{1/k},N^{1/k}+H]$,  $$R_S(F(n))=F(n)\mathfrak S_S(F(n))+O(NL^{-A}),$$ with $O(HL^{-A})$ exceptions. 

In particular, if $m$ is even, then  for almost all $n\in[N^{1/k},N^{1/k}+H]$ such that $F(n)\equiv s_1+s_2 \bmod m$, we have that $F(n)$ is the sum of two primes congruent to $s_1$ and $s_2$ mod $m$ respectively, with $O(HL^{-A})$ exceptions. 

\end{corollary}
\begin{proof}
This follows from the application of Cauchy-Schwarz to the result of Theorem~\ref{cor1}, and by noting that by Lemma~\ref{prod}, if $m$ is even, $\mathfrak S_S(F(n))=0$ unless $F(n)\equiv s_1+s_2\bmod m$.

\end{proof}

As a special case, we obtain the following theorem. 

\begin{theorem}~\label{cor1}
Let $F, s_1, s_2$, and $m$ be as in in Theorem~\ref{estimate}, with $m$ even, and suppose that $F(n)$ takes on a value congruent to $s_1+s_2 \bmod m$. Then there are infinitely many pairs of primes congruent to $s_1$ and $s_2$ modulo $m$, which sum to $F(n)$ for some $n$.
\end{theorem}
\begin{proof}

Since $F$ takes on a value congruent to $s_1+s_2 \bmod m$, there are at least $\frac{H}{m}+O(1)$ values in an interval of size $H$ for which $F$ takes on such a value. Hence, the theorem follows from  Corollary~\ref{cor2}, by taking disjoint intervals with larger and larger $N$.

\end{proof}

This will now give us a proof of our main theorem.
\begin{theorem}
There are infinitely many imaginary quadratic fields with 2-class group of type $(2,2^\ell)$, for any positive integer $\ell$. 

In particular, if $\text{Cl}_2(-d)$ is the 2-class group of $\mathbb Q(\sqrt{-d})$ and $\ell>1$, then
$$\left|\left\{ d\le X, \text{Cl}_2(-d)\cong(2,2^\ell)\right\}\right|\gg \frac{X^{\frac{1}{2}+\frac{1}{2\cdot2^\ell}}}{\log^2 X}.$$

\end{theorem}

\begin{proof}

For $\ell=1$, this was proven in Proposition \ref{k1}. 

By Proposition~\ref{alg}, we have that
$$\left|\left\{ d\le X: \text{Cl}_2(-d)\cong(2,2^\ell)\right\}\right|
\gg|\{p_1,p_2\le X^{1/2}:  p_1+p_2=F(n), p_1\equiv 7\bmod 24, p_2\equiv 11 \bmod 24\}|,
$$
where $F(x)=2(3(2x+1)^2)^{2^{\ell-1}}.$

Since $F(n)\equiv 18\bmod 24$ for all $n\in\mathbb Z$, $R_S(F(n))$ counts such primes, so for some constant $c>0$, $$|\{p_1,p_2\le X^{1/2}:  p_1+p_2=F(n), p_1\equiv 7\bmod 24, p_2\equiv 11 \bmod 24\}|\gg \log^{-2} (X) \sum_{n\le cX^{1/2k}} R_{S}(F(n)).$$

Let $m=24$, $s_1=7$, and $s_2=11$. By combining Corollary~\ref{cor2}, 

with the fact that $F(n)\gg N$ for $n\ge N^{1/k}$, and $\mathfrak S_{S}(F(n))\gg 1$ by Lemma~\ref{prod},  we get that 
$$\sum_{N^{1/k}\le n \le N^{1/k}+H}R_S(F(n))\gg HN $$ Choose $\epsilon$ to be small enough so that $H=N^{1/k-\epsilon}\ge N^{1/k}-1$ for $N\le X$. 
Then summing over all intervals from $N^{1/k}/2^{i+1}$ to $N^{1/k}/2^{i}$ where $i$ ranges from 0, up to the log base 2 of $N^{1/k}$. 
$$\sum_{ n \le N^{1/k}}R_{S}(F(n) \gg N^{k+1}$$

Thus, we find that $$\sum_{n\le cX^{1/2k}} R_{S}(F(n))\gg X^{\frac{1+k}{2k}}.$$ The theorem then follows from the fact that the degree of $F(x)$ is $2^\ell$.

\end{proof}

We remark that after using the results of Dominguez, Miller and Wong~\cite{DMW}, we can apply the same method of proof as above to get a lower bound for the cyclic 2-class groups in their paper. For $\ell>1$, $$\left|\left\{ d\le X, \text{Cl}_2(-d)\cong(2^\ell)\right\}\right|\gg \frac{X^{\frac{1}{2}+\frac{1}{2\cdot2^\ell}}}{\log^2 X}.$$ 

\subsection{Major arcs}

Our goal here is to estimate $$\sum_{\mathfrak M}=\sum_{N^{1/k}\le n \le N^{1/k}+H}\left| \int_{\mathfrak M}f_S(\alpha)^2e(-F(n)\alpha)d\alpha -F(n)\mathfrak S_{S} (F(n))\right|^2$$

Throughout, we will let $q=q_0d$ such that $(q,m)=d$. To  keep things from getting too cumbersome, we will let $h=F(n)$, and define two functions $$v(\beta)=\sum_{n=1}^{c_1N} e(\beta n)\text{ and  } \mu_s(q,a)=\sum_{\substack{r=1 \\ (r,q)=1 \\ r\equiv s \bmod d}}^qe\left(\frac{ar}{q}\right).$$ Note that $h= O(N)$. We will also use the Iverson bracket, defined by
$$\displaystyle[P]=\begin{cases}1\text{ if P is true}\\0\text{ if P is false.}\end{cases}$$

We now prove a lemma which gives a good estimate for $f_s(\alpha)$.

\begin{lemma}\label{est1}

If $1\le a \le q$, and $\alpha\in\mathfrak M'(q,a)$, then there is a positive constant $C$ such that $$f_s\left(\alpha\right)=\frac{v(\alpha-a/q)\mu_s(q,a)}{\phi(q_0m)}+ O(N \exp(-CL^{1/2})).$$

\end{lemma}

\begin{proof}

We start by considering $f_s$ at rational points, and notice that
 $$f_s\left(\frac{a}{q}\right)=\sum_{\substack{r=1 \\ (r,q)=1}}^qe\left(\frac{ar}{q}\right)\vartheta_s(c_1N,q,r)+O(L(\log q)),$$
 
 where $$\vartheta_s(x,q,r)=\sum_{\substack{p\le x \\ p\equiv r \bmod q \\ p\equiv s \bmod m}}\log p$$ is a sum over primes $p$. We can apply Siegel-Walfisz~\cite{SW} to discover that

\begin{eqnarray*}
\vartheta_s(x,q,r)=[s\equiv r \bmod d]\sum_{\substack{p\le x \\ p \equiv r' \bmod \frac{qm}{d}}}\log p
=\frac{x}{\phi(q_0m)}[s\equiv r \bmod d]+ O(N \exp(-C_1 L^{1/2})),
\end{eqnarray*}
for a constant $C_1$ from the Siegel-Walfisz theorem, and where $r'$ is some number coming from the Chinese Remainder Theorem. Thus,
\begin{eqnarray*}f_s\left(\frac{a}{q}\right)&=&\frac{c_1N}{\phi(q_0m)}\sum_{\substack{r=1 \\ (r,q)=1 \\ r\equiv s \bmod d}}^qe\left(\frac{ar}{q}\right)+ O(N \exp(-C_1 L^{1/2}))\\
&=&\frac{c_1N\mu_s(q,a)}{\phi(q_0m)}+ O(N \exp(-C_1 L^{1/2}))
\end{eqnarray*} Following Vaughan~\cite[\S 3]{V}, we can extend this to a general $\alpha\in\mathfrak M'(q,a)$, which gives us the desired result.

\end{proof}

We now define our singular series, the finite version $$\mathfrak S_{s_1s_2}(h;P)=\sum_{q=1}^P \sum_{\substack{a=1 \\ (a,q)=1}}^q \frac{ \mu_{s_1}(q,a)\mu_{s_2}(q,a)}{\phi(q_0m)^2} e\left(\frac{-ah}{q}\right),$$ and its infinite limit $$\mathfrak S_{s_1s_2}(h)=\sum_{q=1}^\infty \sum_{\substack{a=1 \\ (a,q)=1}}^q \frac{ \mu_{s_1}(q,a)\mu_{s_2}(q,a)}{\phi(q_0m)^2} e\left(\frac{-ah}{q}\right).$$

In Lemma~\ref{convergence} we will prove that the infinite singular series converges, and in Lemma~\ref{prod} we will find a product expansion for it. But first, we will need the following lemma about $\mu_s(q,a)$.

\begin{lemma}\label{mu}

Let $(q,a)=1$. Then,
$$\mu_s(q,a)=\sum_{\substack{r=1 \\ (r,q)=1 \\ r\equiv s \bmod d}}^qe\left(\frac{ar}{q}\right)=\mu(q_0)e\left(\frac{as}{q}\right)e\left(\frac{as'}{q_0}\right)[(q_0,m)=1]$$

where $s'$ is chosen such that $s'd\equiv -s \bmod p_i$ for every prime $p_i$ dividing $q$, but not dividing $d$.
\end{lemma}
\begin{proof}
We first let $\mu_s(q):=\mu_s(q,1)$, and notice that $\mu_{as}(q)=\mu_s(q,a)$. 

Let $q=p_1^{e_1}\cdots p_n^{e_n}$, where $e_i\ge1$ for all $i$. Also let $0\le b_i\le e_i$ such that $d=p_1^{b_1}\cdots p_n^{b_n}$. 

Arrange (WLOG) the primes so that the first $\gamma$ of the $p_i$'s are the primes not dividing $d$ (i.e. $b_i=0$ for $i\le\gamma$).We let $d_i^{-1}$ such that $d_i^{-1}d\equiv 1 \bmod p_i$, and 

By the inclusion-exclusion principle, we have the following.

$$\mu_s(q)=\sum_{\substack{r=1\\ r\equiv s \bmod d}}^q e\left(\frac{r}{q}\right)+\sum_{k=1}^\gamma(-1)^k\left[\sum_{1\le i_1<\cdots<i_k\le\gamma}\sum_{\substack{t=1\\ t\equiv -sd_{i_j}^{-1} \left({p_{i_j}}\right)}}^{q_0} e\left(\frac{td+s}{q}\right)\right].$$ 

By the Chinese Remainder Theorem, we can combine the congruence conditions in the innermost sum into one, $t\equiv s'[i_1,\cdots, i_k] \bmod p_{i_1}\cdots p_{i_k}$, for some $s'[i_1,\cdots,i_k]$.   

So, the innermost sum can be evaluated as
\begin{eqnarray*}\sum_{\substack{t=1\\ t\equiv s'[i_1,\cdots,i_k] \left({p_{i_1}\cdots p_{i_k}}\right)}}^{q_0} e\left(\frac{td+s}{q}\right)&=&e\left(\frac{s}{q}\right)\sum_{u=1}^{q_0/(p_{i_1}\cdots p_{i_k})}e\left(\frac{up_{i_1}\cdots p_{i_k}+s'[i_1,\cdots,i_k]}{q_0}\right)\\
&=&e\left(\frac{s}{q}\right)e\left(\frac{s'[i_1,\cdots,i_k]}{q_0}\right)[q_0/(p_{i_1}\cdots p_{i_k})=1]. 
\end{eqnarray*}

Hence, the entire sum is $$\mu_s(q)=\mu(\rho)e\left(\frac{s}{q}\right)e\left(\frac{s'[1,\cdots,\gamma]}{q_0}\right)[q_0=\rho],$$ where $\rho=p_1\cdots p_\gamma$. The theorem then follows.

\end{proof}

Using the previous lemma, we may now prove that the singular series converges, and has a nice expression in terms of multiplicative functions. 

\begin{lemma}\label{convergence}

The sum $\mathfrak S_{s_1s_2}(h)$ converges, and  $$\mathfrak S_{s_1s_2}(h)-\mathfrak S_{s_1s_2}(h;P)\ll \frac{h\tau(h)}{P\phi(h)},$$ where $\tau$ is the sum of divisors function. Furthermore, $$\mathfrak S_{s_1s_2}(h)=\sum_{\substack{q=1\\ (q_0,m)=1}}^\infty \frac{\mu(q_0)^2c_{q_0}(h)c_d(s_1+s_2-h)}{\phi(q_0m)^2},$$ and $\mathfrak S_{s_1s_2}(h;P)\ll L$.

\end{lemma}
\begin{proof}


We notice that this sum is similar to the Ramanujan sum $$c_q(n)=\sum_{\substack{a=1\\ (a,q)=1}}^q e\left(\frac{an}{q}\right)$$ and by Lemma~\ref{mu} we can rewrite it using Ramanujan sums:
\begin{eqnarray*}
&&\sum_{\substack{a=1 \\ (a,q)=1}}^q \mu_{s_1}(q,a)\mu_{s_2}(q,a) e\left(\frac{-ah}{q}\right)\\
&=&\sum_{\substack{a=1 \\ (a,q)=1}}^q e\left(\frac{a(s_1+s_2-h)}{q}\right)e\left(\frac{as_1'+as_2'}{q_0}\right)\mu(q_0)^2[(q_0,m)=1]\\
&=&c_q(s_1+s_2-h+d(s_1'+s_2'))\mu(q_0)^2[(q_0,m)=1]\\
&=&c_{q_0}(-h)c_d(s_1+s_2-h)\mu(q_0)^2[(q_0,m)=1],
\end{eqnarray*} so we have that 
$$\mathfrak S_{s_1s_2}(h;P)=\sum_{\substack{q=1\\ (q_0,m)=1}}^P\frac{\mu(q_0)^2c_{q_0}(-h)c_d(s_1+s_2-h)}{\phi(q_0m)}.$$

Now using the fact that $$c_q(-h)=\phi(q)\mu\left(\frac{q}{(q,h)}\right)\phi\left(\frac{q}{(q,h)}\right)^{-1},$$
we consider the following bound based on equation (3) of~\cite{P},
\begin{eqnarray*}\sum_{\substack{q>P\\ (q_0,d)=1}} \frac{\mu(q_0)^2c_{q_0}(-h)c_d(s_1+s_2-h)}{\phi(q_0m)^2}\ll\sum_{q>P} \frac{c_{q}(h)}{\phi(q)^2} \ll \sum_{q>P}\phi(q)^{-1}\phi\left(\frac{q}{(q,h)}\right)^{-1}\\
\ll\sum_{d|h}\phi(d)^{-1}\sum_{r>P/d}\phi(r)^{-2}\ll P^{-1}\sum_{d|h}\frac{d}{\phi(d)}\ll \frac{h\tau(h)}{P\phi(h)}.
\end{eqnarray*}  This proves that $$\mathfrak S_{s_1s_2}(h)=\sum_{\substack{q=1\\ (q_0,m)=1}}^\infty \frac{\mu(q_0)^2c_{q_0}(h)c_d(s_1+s_2-h)}{\phi(q_0m)^2}$$ converges.

Finally, from our above expression $\mathfrak S_{s_1s_2}(h;P)\ll \sum_{q=1}^P \frac{1}{\phi(q)}\ll L.$

\end{proof}

From this lemma, we may now prove our main theorem concerning the major arcs.

\begin{theorem}\label{major}
Let $q=q_0d$ such that $(q,m)=d$, and let $A,\epsilon\ge 0$, and $H$ such that $N^{1/(3k)+\epsilon}\le H \le N^{1/k-\epsilon}$, then there is some constant $c>0$ such that
\begin{eqnarray*}
\sum_{\mathfrak M}&=&\sum_{N^{1/k}\le n \le N^{1/k}+H}\left| \int_{\mathfrak M}f_S(\alpha)^2e(-h\alpha)d\alpha -h\mathfrak S_{S} (h)\right|^2\\
&\ll& HN^2L^{-2B+c}+HN^2L^A,
\end{eqnarray*}
where $$\mathfrak S_{S}(h)=\mathfrak S_{s_1s_1}(h)+2\mathfrak S_{s_1s_2}(h)+\mathfrak S_{s_2s_2}(h).$$

\end{theorem}

\begin{proof}
Recall that $f_S(\alpha)=f_{s_1}(\alpha)+f_{s_2}(\alpha)$. We can separate the integral and its estimate into three parts, corresponding to the 3 terms in $f_S(\alpha)^2$,  
\begin{eqnarray*}\int_{\mathfrak M}\big(f_{s_1}(\alpha)^2+2f_{s_1}(\alpha)f_{s_2}(\alpha)+f_{s_2}(\alpha)^2\big)e(-h\alpha)d\alpha\\ -h\big(\mathfrak S_{s_1s_1}(h)+2\mathfrak S_{s_1s_2}(h)+\mathfrak S_{s_2s_2}(h)\big).
\end{eqnarray*} It is easier to bound these three parts separately, and we will estimate the contribution from $$\int_{\mathfrak M}f_{s_1}(\alpha)f_{s_2}(\alpha)e(-\alpha n)d\alpha$$ since the other cases follow from this one. To do this, we will mainly use Vaughan's arguments~\cite[\S 3]{V} modified appropriately in a way similar to the modifications in~\cite[\S 3.3]{DMW} and~\cite[\S 2]{P}. 

Applying Lemma~\ref{est1} to the product $f_{s_1}(\alpha)f_{s_2}(\alpha)$, we find that $$f_{s_1}(\alpha)f_{s_2}(\alpha)-\frac{\mu_{s_1}(q,a)\mu_{s_2}(q,a)}{\phi(q_0m)^2}v(\alpha-a/q)^2\ll N^2 \exp(-CL^{1/2}), $$ and integrating over $\mathfrak M$ gives us 

\begin{eqnarray*}\sum_{q\le P}\sum_{\substack{a=1\\ (a,q)=1}}^q\int_{\mathfrak M(q,a)}\left(f_{s_1}(\alpha)f_{s_2}(\alpha)-\frac{\mu_{s_1}(q,a)\mu_{s_2}(q,a)}{\phi(q_0m)^2}v(\alpha-a/q)^2\right)e(-\alpha h)\\
\ll P^3N \exp(-CL^{1/2}).
\end{eqnarray*} By definition of the major arcs, we arrange this as

$$\int_{\mathfrak M}f_{s_1}(\alpha)f_{s_2}(\alpha)e(-\alpha h)=\mathfrak S_{s_1s_2}(h,P)\int_{-P/N}^{P/N} v(\beta)^2e(-\beta h)d\beta + O(P^3N \exp(-CL^{1/2})).$$



According to Vaughan~\cite[Chapter~3]{V}, we have $$\int_{P/N}^{1/2}|v(\beta)|^2d\beta\ll P^{-1}N,$$ and by the definition of $v(\beta)$, the integral $$\int_{-1/2}^{1/2} v(\beta)^2e(-\beta h)d\beta $$ simply counts the number of solutions to $n_1+n_2=h$. Hence it is equal to $h-1$ if $h=F(n)$ is positive. This will clearly be the case if $N$ is sufficiently large, since $F(n)$ has positive leading coefficient.

Combining these with the results from Lemma~\ref{convergence} we can now use our singular series through Perelli's~\cite{P} and Vaughan's~\cite{V} arguments.  $$\int_{\mathfrak M}f_{s_1}(\alpha)f_{s_2}(\alpha)e(-\alpha h)-h\mathfrak S_{s_1s_2}(h) \ll  N\left|\frac{h\tau(h)}{P\phi(h)}\right| + NL^{1-B}+NL^{-A/2}.$$ 

Applying Nair's theorem~\cite{N}, we can bound the sum 
\begin{eqnarray*}\sum_{N^{1/k}\le n \le N^{1/k}+H}\left| \int_{\mathfrak M}f_{s_1}(\alpha)f_{s_2}(\alpha)e(-\alpha h)-h\mathfrak S_{s_1s_2}(h) \right|^2 \ll HN^2L^{-2B+c_2}+HN^2L^{-A},
\end{eqnarray*}
where $c_2>0$ is a constant depending on $m$, $F$ and $N$.


We get similar bounds for the other two cases, differing only by the constant $c_2$ in each case. Thus, we may use the triangle inequality to combine all of these together, giving us that 
\begin{eqnarray*}\sum_{\mathfrak M}=\sum_{N^{1/k}\le n \le N^{1/k}+H}\left| \int_{\mathfrak M}f_S(\alpha)^2e(-h\alpha)d\alpha -h\mathfrak S_{S} (h)\right|^2
\ll HN^2L^{-2B+c}+HN^2L^{-A},
\end{eqnarray*} for some constant $c>0$ depending on $m$, $F$ and $N$. We note that $c$ is ineffective because of the use of Siegel-Walfisz.

\end{proof}

We will also require the following important lemma, which gives a product expansion for the singular series. 

\begin{lemma}\label{prod}

We have the following product expansion 

$$\mathfrak S_{s_1s_2}(h)=[s_1+s_2\equiv h\bmod m]\frac{m}{\phi(m)^2}\prod_{\substack{p\nmid h\\ p\nmid m}}\left(1-\frac{1}{(p-1)^2}\right)\prod_{\substack{p|h\\ p\nmid m}}\left(1+\frac{1}{p-1}\right).$$

Furthermore, if $F(n)$ is an even value congruent to $s_1+s_2\bmod m$, then
 $$\mathfrak S_{S}(F(n))\gg 1.$$

\end{lemma}
\begin{proof}
Using the expression in Lemma~\ref{convergence}, and the multiplicative properties of the arithmetic functions involved, we get
\begin{eqnarray*}
\mathfrak S_{s_1s_2}(h)&=&\sum_{d|m}\frac{c_d(s_1+s_2-h)}{\phi(m)^2}\sum_{\substack{q=1\\ (q,m)=1}}^\infty \frac{\mu(q)^2c_{q}(h)}{\phi(q)^2}\\
&=&[s_1+s_2\equiv h\bmod m]\frac{m}{\phi(m)^2}\prod_{\substack{p\nmid h\\ p\nmid m}}\left(1-\frac{1}{(p-1)^2}\right)\prod_{\substack{p|h\\ p\nmid m}}\left(1+\frac{1}{p-1}\right).
\end{eqnarray*} In particular, note that $\mathfrak S_{s_1s_2}(h)=0$ iff $s_1+s_2\not\equiv h \bmod m$, or if $m$ and $h$ are both odd. This makes sense since we are trying to count pairs of primes congruent to $s_1$ and $s_2$ modulo $m$ that sum to $h$. But as long as it is not zero, we have that $\mathfrak S_{s_1s_2}(h)\gg 1$. Our condition on $F(n)$ forces it to not be zero, hence the lemma is proved. 

We remark that when $m=1$, this becomes the singular series for the binary Goldbach problem in~\cite{P} or~\cite{V}, and that when $m=8$, with $s_1=3$, $s_2=5$, then $\mathfrak S_S$ is the singular series in~\cite[Equation (34)]{DMW}.

\end{proof}

\subsection{Minor arcs}

Here, our goal is to bound $$\sum_{\mathfrak m}=\sum_{N^{1/k}\le n \le N^{1/k}+H}\left| \int_{\mathfrak m}f_S(\alpha)^2e(-F(n)\alpha)d\alpha\right|^2.$$ This will rely heavily on Perelli's arguments in~\cite{P} and~\cite{PP}, with the changes similar to those provided in Dominguez, Miller, and Wong~\cite[pp. 12-13]{DMW}.

Following Perelli, we let $Q'=H^k L^{-B/4}$, and $Q=\frac{Q'^{1/2}}{2}$, and we let $\mathfrak M(q,a)$ and $\mathfrak {\overline{M}}(q,a)$ be the Farey arcs with center $\frac{a}{q}$ of the Farey dissections of order $Q$ and $Q'$ respectively. 

We let $$\mathfrak{\overline{M}}=\bigcup_{q\le L^{B/4}}\bigcup_{\substack{a=1 \\ (a,q)=1}}^q\mathfrak{\overline{M}}(q,a),$$ and let $\mathfrak{\overline{ m}}= [0,1]\backslash \mathfrak{\overline{ M}}$.  


We are now ready to prove the following bound for the minor arcs. 
\begin{theorem}\label{minor} For $B>0$ large enough, 

$$\sum_{\mathfrak m} \ll HN^2L^{-B}.$$

\end{theorem}

\begin{proof}




Following Perelli's arguments~\cite[Equations (5)-(11)]{P}, we find that by a variant of Weyl's inequality $$\sum_{\mathfrak m}\ll HNL^{B/2+1}\sup_{\xi\in\mathfrak m}\max_{\substack{\overline q < L^{B/4}\\ (\overline a, \overline q)=1}}\int_{{\mathfrak m}\cap(\xi+\overline{\mathfrak M}(\overline q, \overline a))} |f_S(\alpha)|^2 d\alpha + HN^2L^{-(B-4k^2-2^{k+4})/2^{k+3}}.$$ 


Manipulating the arcs as in Perelli~\cite[Equations (12)-(14)]{P} we then get $$\sup_{\xi\in\mathfrak m}\max_{\substack{\overline q \le L^{B/4}\\ (\overline a, \overline q)=1}}\int_{{\mathfrak m}\cap(\xi+\overline{\mathfrak M}(\overline q, \overline a))} |f_S(\alpha)|^2d\alpha \ll \max_{\substack{q\le Q\\ (a,q)=1}}\int_{\mathfrak M''(q,a)}|f_S(\alpha)|^2d\alpha,$$
 where $$\mathfrak M''(q,a)=\begin{cases} \mathfrak M(q,a)\backslash \mathfrak M'(q,a), \text{ if } q\le P,\\ \mathfrak M(q,a), \text{ if } P\le q \le Q.\end{cases}$$

We will now examine $f_S$ and rewrite it in terms of other functions, and center it at $\frac{a}{q}$. First, consider the Dirichlet characters with modulus $m$. There are $t=\phi(m)$ of these, and by orthogonality of characters, we can take a linear combination such that $$[n\equiv s_1,s_2\bmod m]=a_1\chi_1(n)+\cdots + a_t\chi_t(n).$$ We then write $$f_S\left(\frac{a}{q}+\eta\right)=\frac{\mu_{s_1}(q,a)+\mu_{s_2}(q,a)}{\phi(q)}T(\eta)+R(\eta,q,a),$$ where
$$T(\eta)=\sum_{n\le c_1 N} e(n\eta), \quad R_S(\eta,q,a)=\frac{1}{\phi(q)}\sum_{\chi\bmod q}\chi(a)\tau(\overline \chi)W_S(\chi,\eta)+O(N^{1/2}),$$

$$W_S(\chi,\eta)=\sum_{\substack{n\le c_1 N\\ n\equiv s_1,s_2\bmod m}}\Lambda(n)\chi(n)e(n\eta)-\sum_{i=1}^ta_i[\chi=\overline \chi_i]T(\eta),$$ and $\tau(\chi)$ is the Gauss sum for characters with conductor $q$. We note that we can lift all these characters to characters modulo $qm$, and the comparison in $W_S$ is over these lifted characters. Also, note that the difference between $\log p$ and $\Lambda(n)$ gets absorbed into the error term $O(N^{1/2})$.

By Lemma~\ref{mu}, we have that $\mu_{s_1}(q,a)+\mu_{s_2}(q,a)\le 2$, hence we have $$\int_{\mathfrak M''(q,a)}|f_S(\alpha)|^2d\alpha \ll \frac{1}{\phi(q)^2}\int_{\xi(q)}|T(\eta)|^2d\eta+\int_{\frac{-1}{qQ}}^{\frac{1}{qQ}}|R_S(\eta,q,a)|^2d\eta,$$ where $$\xi(q)=\begin{cases}\left(\frac{L^B}{N},\frac{1}{2}\right), \text{ if } q\le L^B,\\
\left(-\frac{1}{qQ},\frac{1}{qQ}\right),\text { if } L^B<q\le Q.\end{cases}$$ Using the fact that $T(\eta)$ is a geometric series, we can see  that $T(\eta)\ll \min\{N,1/||\eta||\},$ thus $$\frac{1}{\phi(q)^2}\int_{\xi(q)}|T(\eta)|^2d\eta\ll NL^{-B}.$$

Following Perelli's distinction between good and bad characters~\cite[Equations (16)-(24)]{P}, we can directly use his argument for $R$ to give us the bound
$$\int_{\frac{-1}{qQ}}^{\frac{1}{qQ}}|R_S(\eta,q,a)|^2d\eta\ll \frac{q}{\phi(q)}\sum_{\chi\text{ good}}\int_{\frac{-1}{qQ}}^{\frac{1}{qQ}}|W_S(\chi,\eta)|^2d\eta+NL^{-B}.$$ 

We now consider $W_S$ more carefully:

\begin{eqnarray*}W_S(\chi,\eta)&=&\sum_{\substack{n\le c_1 N\\ n\equiv s_1,s_2\bmod m}}\Lambda(n)\chi(n)e(n\eta)-\sum_{i=1}^ta_i[\chi=\overline \chi_i]T(\eta)\\
&=&\sum_{n\le c_1 N}\left[\Lambda(n)\chi(n)e(n\eta)(a_1\chi_1(n)+\cdots+a_t\chi_t(n))-\sum_{i=1}^ta_i[\chi=\overline \chi_i]e(n\eta)\right]\\
&=&\sum_{n\le c_1 N}\left[ \sum_{i=1}^ta_i\left(\Lambda(n)\chi(n)\chi_i(n)-[\chi=\overline\chi_i]\right)\right]e(n\eta)
\end{eqnarray*} So by Gallagher's Lemma~\cite[Lemma 1]{G}, we get 

$$\int_{\frac{-1}{qQ}}^{\frac{1}{qQ}}|W_S(\chi,\eta)|^2d\eta \ll \frac{1}{qQ^2}\int_{\frac{-qQ}{2}}^{c_1N}\left|\sum_{\substack{n=x\\ n\in[1,c_1N]}}^{x+\frac{qQ}{2}}\left(\sum_{i=1}^ta_i\left(\Lambda(n)\chi(n)\chi_i(n)-[\chi=\overline\chi_i]\right)\right)\right|^2dx.$$ Now, we have the explicit formula $$\sum_{n\le x}\Lambda(n)\chi(n)\chi_i(n) -[\chi=\overline\chi_i]x =-\sum_{|\gamma|\le c_1N}\frac{x^\rho}{\rho}+O(L^2),$$ for $4\le x\le c_1N$, and $qm\le c_1N$,  where $\rho=\beta+i\gamma$ are zeros of $L(s,\chi\chi_i)$ with $0<\beta<1$.
Using this, we argue as in Perelli and Pintz  in~\cite[Equations (22)-(26)]{PP} for their estimate of $W_2$, to get that
$$\frac{1}{qQ^2}\int_{\frac{-qQ}{2}}^{c_1N}\left|\sum_{\substack{n=x\\ n\in[1,c_1N]}}^{x+\frac{qQ}{2}}\left(\sum_{i=1}^ta_i\left(\Lambda(n)\chi(n)\chi_i(n)-[\chi=\overline\chi_i]\right)\right)\right|^2dx \ll NL^{-B},$$ where we use the inequality $|a+b|^2\ll |a|^2+|b|^2$ to handle the extra sum.

Piecing these all together, we find that by choosing $B$ sufficiently large relative to $k$, we have

$$\sum_{\mathfrak m}\ll HN^2L^{-(B-4k^2-2^{k+4})/2^{k+3}}+HN^2L^{1-B/2}\ll HN^2L^{-B}.$$
\end{proof}

\bibliography{Ref_ALopez}{}
\bibliographystyle{plain}
 \end{document}